\documentclass[a4paper,12pt,leqno]{article}

\usepackage{amssymb,amsmath}
\usepackage[abbrev]{amsrefs}
\usepackage{xy}
\xyoption{all}

\newtheorem{defn}{Definition}[section]
\newtheorem{proposition}[defn]{Proposition}
\newtheorem{corollary}[defn]{Corollary}
\newtheorem{rem}[defn]{Remark}
\newtheorem{exm}[defn]{Example}
\newtheorem{lemma}[defn]{Lemma}
\newtheorem{theorem}[defn]{Theorem}
\newtheorem{notat}[defn]{Notation}
\newtheorem{newpar}[defn]{}

\newtheorem{xdefn}{Definition.}
\newtheorem{xproposition}{Proposition.}
\newtheorem{xcorollary}{Corollary.}
\newtheorem{xrem}{Remark.}
\newtheorem{xexm}{Example.}
\newtheorem{xlemma}{Lemma.}
\newtheorem{xtheorem}{Theorem.}
\newtheorem{xnotat}{Notation.}
\newtheorem{xnewpar}{\it}
\newtheorem{xproof}{{\it Proof. }}
\newtheorem{xproofsketch}{{\it Proof sketch. }}
\newtheorem{xproofof}{{\it Proof}}

\newenvironment{definition}{\begin{defn}\em}{\end{defn}}

\newenvironment{proof}{\begin{xproof}\em}{\end{xproof}}

\newenvironment{newparagraph*}[1]{\begin{xnewpar}\hspace*{-1.5mm}{#1}. \rm}{\end{xnewpar}}

\newenvironment{definition*}{\begin{xdefn}\em}{\end{xdefn}}
\newenvironment{remark*}{\begin{xrem}\em}{\end{xrem}}
\newenvironment{example*}{\begin{xexm}\em}{\end{xexm}}
\newenvironment{notation*}{\begin{xnotat}\em}{\end{xnotat}}
\newenvironment{proposition*}{\begin{xproposition}}{\end{xproposition}}
\newenvironment{corollary*}{\begin{xcorollary}}{\end{xcorollary}}
\newenvironment{lemma*}{\begin{xlemma}}{\end{xlemma}}
\newenvironment{theorem*}{\begin{xtheorem}}{\end{xtheorem}}

\def\qed{\hspace{0.3cm}{\rule{1ex}{2ex}}}
\newcommand\V{\bigvee}
\newcommand\ie{i.e.}
\newcommand\eg{e.g.}
\newcommand\st{\mid}
\newcommand\cf{\textrm{cf.}}
\newcommand\downsegment{{\downarrow}}

\newcommand\pwset[1]{\wp(\,#1)}
\newcommand\opens{\operatorname{\mathcal{O}}}

\newcommand\topology{\operatorname{\Omega}}
\newcommand\groupoid{\operatorname{\mathcal{G}}}

\newcommand\Frm{\mathit{Frm}}
\newcommand\Loc{\mathit{Loc}}

\newcommand\Top{\mathit{Top}}
\newcommand\opp[1]{{#1}^{\textrm{op}}}

\newcommand\ident{\mathrm{id}}
\newcommand\ipi{\mathcal I}

\newcommand\SL{\mathit{SL}}

\newcommand\lcc{\operatorname{{\mathcal L}^{\vee}}}

\newcommand\hide[1]{}

\newcommand\sections{\mathit{\Gamma}}

\newcommand\IQFrm{\mathit{IQFrm}}
\newcommand\InvQuF{\mathit{InvQuF}}
\newcommand\IQLoc{\mathit{IQLoc}}
\newcommand\Grpd{\mathit{Gpd}}

\newcommand\IQLocLax{\mathit{IQLoc}_\ell}

\newcommand\CGrpd{\mathit{Gpd}_{\!\mathit{C}}}
\newcommand\AGrpd{\mathit{Gpd}_{\!\mathit{A}}}
\newcommand\biGrpd{\mathbf{Gpd}}
\newcommand\biIQLoc{\mathbf{IQLoc}}

\newcommand\tp{\operatorname\otimes}
\newcommand\lpl{\operatorname{\mathit{P}_{\mathit L}}}

\newcommand\act{\cdot}

\newenvironment{eq}{\setcounter{equation}{\arabic{defn}}\begin{equation}}{\end{equation}\setcounter{defn}{\arabic{equation}}}
\newenvironment{eqarray}{\setcounter{equation}{\arabic{defn}}\begin{eqnarray}}{\end{eqnarray}\setcounter{defn}{\arabic{equation}}}

\begin{document}

\title{Functoriality of groupoid quantales. I\thanks{Partially funded by FCT/Portugal through projects EXCL/MAT-GEO/0222/2012 and PEst-OE/EEI/LA0009/2013.}}

\author{\sc Pedro Resende}

\date{~}

\maketitle

\begin{abstract}
We provide three functorial extensions of the equivalence between localic \'etale groupoids and their quantales. The main result is a biequivalence between the bicategory of localic \'etale groupoids, with bi-actions as 1-cells, and a bicategory of inverse quantal frames whose 1-cells are bimodules. As a consequence, the category $\InvQuF$ of inverse quantale frames, whose morphisms are the (necessarily involutive) homomorphisms of unital quantales, is equivalent to a category of localic \'etale groupoids whose arrows are the algebraic morphisms in the sense of Buneci and Stachura.
We also show that the subcategory of $\InvQuF$ with the same objects and whose morphisms preserve finite meets is dually equivalent to a subcategory of the category of localic \'etale groupoids and continuous functors whose morphisms, in the context of topological groupoids, have been studied by Lawson and Lenz. \\
\vspace*{-2mm}~
\vspace*{-2mm}~
\end{abstract}

{\small{\tableofcontents}}

\section{Introduction}\label{introduction}

Locales \cite{stonespaces} are a point-free version of topological spaces. An example is the locale $I(A)$ of closed ideals of an abelian C*-algebra $A$, which is an algebraic (lattice-theoretic) object that contains all the information about the spectrum of the algebra. In many contexts locales are more convenient to work with than spaces, especially when points, separation axioms, etc., can be ignored. In such situations locales often lead to more general theorems, in particular theorems that are constructive in the sense of being valid in arbitrary toposes \cite{pointless}. One can also think of a locale as being a kind of commutative ring (with the underlying abelian group replaced by a sup-lattice). The similarity to commutative algebra goes a long way and it is at the basis of the groupoid representation of Grothendieck toposes \cite{JT}, in which localic groupoids (\ie, groupoids in the category of locales $\Loc$) arise from toposes via descent.

A generalization of locales is given by quantales \cite{Rosenthal1}, which are semigroups in the category of sup-lattices and thus are like noncommutative rings. The idea that some quantales can be regarded as generalized, and C*-algebra related, point-free spaces has been around since the term ``quantale'' was coined \cite{M86,BRB , Rosicky ,MP1, MP2,K02,KR}, and there is a particularly good interplay between quantales and groupoids \cite{Re07,PR12,PaRe13}. Concretely, the quantale of a topological groupoid $G$ (with open domain map) is the topology of the arrow space $G_1$ equipped with pointwise operations of multiplication and involution. This can be regarded as a convolution ``algebra'', for if we identify each open subset $U\subset G_1$ with a continuous mapping to Sierpi{\' n}ski space $\widetilde U:G_1\to\$$ we obtain \[\widetilde{UV}=\widetilde U*\widetilde V\;,\] where the convolution of two continuous maps $\phi,\psi:G_1\to\$$ is defined by
\[
\phi*\psi(g) = \V_{g=hk} \phi(h)\psi(k)\;.
\]
This construction can be carried over to localic groupoids, and the resulting correspondence between groupoids and quantales restricts to a bijection between localic \'etale groupoids (up to isomorphisms) and inverse quantal frames \cite{Re07}. This is a topological analogue of the dualities of algebraic geometry, with \'etale groupoids playing the role of ``noncommutative varieties''. In particular, any Grothendieck topos coincides, at least in the case of the topos of an \'etale groupoid, with a category of modules over the quantale of the groupoid \cite{GSQS} (see also \cite{HS3} for other quantale representations of Grothendieck toposes). However, this analogy is objects-only because the bijection is not functorial with respect to groupoid functors and quantale homomorphisms, and the main aim of this paper is to address this issue.

This functoriality problem is similar to another, well known, one: locally compact groupoids \cite{RenaultLNMath,Paterson} generalize both locally compact groups and locally compact spaces but, if we take groupoid morphisms to be general functors, this generalization is not functorial with respect to convolution algebras and their homomorphisms. In order to see this it suffices to notice that Gelfand duality yields a contravariant functor from compact Hausdorff spaces to C*-algebras, whereas the universal C*-algebra of a discrete group defines a covariant functor. An interpretation of this discrepancy is that a groupoid C*-algebra can be regarded as a description of the space of orbits (in a generalized sense) of the groupoid and that groupoid functors fail to account for this \cite{Connes}. In addition, for two such spaces to be considered ``the same'' one usually requires the algebras to be only Morita equivalent rather than isomorphic. Accordingly, appropriate definitions of morphism for groupoids, which subsume groupoid functors and map functorially to bimodules, have been defined in terms of bi-actions \cite{HS87,Mr99,MRW87,La01}.
The idea of a groupoid as a generalized space of orbits is even more explicit in topos theory, since any Grothendieck topos is, in a suitable sense, a quotient of the object space of a groupoid in the 2-category of toposes and geometric morphisms \cite{Moer88}. Again, morphisms can be taken to be bi-actions \cite{Bunge, Moer90, Moer87}.

In the present paper we show that the correspondence between groupoids and quantales is functorial in the bicategorical sense suggested by the above remarks.
In order to achieve this we show, in section \ref{funct2}, following preliminary results about groupoid actions in section \ref{gpdactions}, that the bi-actions of localic \'etale groupoids map functorially to quantale bimodules, and that, improving on what would be expected for convolution algebras, this assignment restricts to a biequivalence, namely between the bicategory $\biGrpd$ of localic \'etale groupoids and a bicategory $\biIQLoc$ of inverse quantal frames.

As an example, at the end of section \ref{funct2} we discuss the notion of algebraic morphism of groupoids \cite{BS05,Bun08}. Algebraic morphisms are examples of groupoid bi-actions that map functorially and covariantly to homomorphisms of C*-algebras \cite{BS05} and to homomorphisms of inverse semigroups \cite{BEM12}, and furthermore, as noted in \cite{BS05}, specialize both to group homomorphisms (covariantly) and to continuous maps between topological spaces (contravariantly), hence in a narrower extent suggesting a solution to the functoriality problem addressed in this paper. A corollary of our bicategorical equivalence is that the algebraic morphisms of \'etale groupoids are ``the same'' as the homomorphisms of unital (involutive) quantales between the quantales of the groupoids and yield a category $\AGrpd$ which is equivalent to the category $\InvQuF$ of \cite{Re07}.
For \'etale groupoids all the above remarks (except those pertaining to C*-algebras) follow readily from this identification. Another consequence is that there is a covariant functor from a non-trivial category of quantales to C*-algebras. The existence of such a functor is interesting in its own right, in view of the difficulties that arise with respect to functoriality when studying correspondences between quantales and C*-algebras \cite{KPRR,KR}.

In addition to the above results, and independently from bi-actions and bimodules, we show, in section \ref{funct1}, that the subcategory $\IQFrm$ of $\InvQuF$ whose morphisms are also locale homomorphisms is dually equivalent to a category $\CGrpd$ of \'etale groupoids whose morphisms, in the topological context, coincide with the covering functors \cite{LL}.

Other functorial aspects of groupoid quantales, for instance regarding Hilsum--Skandalis maps and Morita equivalence, will be addressed in a subsequent paper.

\section{Preliminaries}\label{groupoidsandquantales}

In this section we introduce basic facts, terminology and notation for sup-lattices, locales and groupoid quantales, mostly following \cite{JT,stonespaces,Re07,GSQS}.

\subsection{Locales}

By a \emph{sup-lattice} is meant a complete lattice, and a \emph{sup-lattice homomorphism} $h:Y\to X$ is a mapping that preserves arbitrary joins. The resulting \emph{category of sup-lattices} $\SL$ is bi-complete and monoidal \cite{JT}. The top element of a sup-lattice $X$ is denoted by $1_X$ or simply $1$, and the bottom element by $0_X$ or simply $0$.
A sup-lattice further satisfying the infinite distributive law
\[
x\wedge\V_{\alpha} y_\alpha = \V_{\alpha} x\wedge y_\alpha
\]
is a \emph{frame}, or \emph{locale},
and a \emph{frame homomorphism} $h:Y\to X$ is a sup-lattice homomorphism that preserves finite meets. This defines the \emph{category of frames}, $\Frm$.
The dual category $\Loc=\opp\Frm$ is referred to as the \emph{category of locales} \cite{stonespaces}, and its arrows are called \emph{continuous maps}, or simply \emph{maps}. These categories are bi-complete, and the product of $X$ and $Y$ in $\Loc$ is denoted by $X\tp Y$, since it coincides with the tensor product in $\SL$ \cite[\S I.5]{JT}.
The coproduct of $X$ and $Y$ in $\Loc$ is the direct sum in $\SL$ and we denote it by $X\oplus Y$.

A \emph{subframe} of $X$ is a subset closed under finite meets and arbitrary joins, whereas a \emph{sublocale} is an equivalence class of a surjective frame homomorphism, or, equivalently, a \emph{nucleus} on $X$, by which is meant a closure operator $j$ on $X$ satisfying the law $j(x\wedge y)=j(x)\wedge j(y)$. An example is the \emph{open sublocale} associated to an element $s\in X$, which corresponds to the frame surjection
$X\to \downsegment (s)=\{x\in X\st x\le s\}$
defined by $x\mapsto x\wedge s$.

If $f:X\to Y$ is a map of locales we refer to the corresponding frame homomorphism $f^*:Y\to X$ as its \emph{inverse image}. Such a homomorphism turns $X$ into a \emph{$Y$-module} (in the sense of quantale modules --- see \ref{quantalesandmodules}) with action $y\cdot x=f^*(y)\wedge x$ for all $y\in Y$ and $x\in X$, and the map $f$ is \emph{open} if $f^*$ has a left adjoint $f_!:X\to Y$, referred to as the \emph{direct image} of $f$, which is a homomorphism of $Y$-modules. A \emph{local homeomorphism} $f:X\to Y$ is a (necessarily open) map for which there is a subset $\sections\subset X$ satisfying $\V \sections =1$ (a \emph{cover} of $X$) such that for each $s\in \sections$ the direct image $f_!$ restricts to an isomorphism $\downsegment (s)\cong\downsegment (f_!(s))$.
Both open maps and local homeomorphisms are stable under pullbacks.

\subsection{Groupoids}

A \emph{localic groupoid} is an internal groupoid in $\Loc$. We denote the locales of objects and arrows of a localic groupoid $G$ respectively by $G_0$ and $G_1$, and adopt the following notation for the structure maps,
\[G\ \ \ \ =\ \ \ \ \xymatrix{
G_2\ar[rr]^-m&&G_1\ar@(ru,lu)[]_i\ar@<1.2ex>[rr]^r\ar@<-1.2ex>[rr]_d&&G_0\ar[ll]|u
}\;,\]
where $G_2$ is the pullback of the \emph{domain} and \emph{range} maps:
\[\vcenter{\xymatrix{G_2\ar[r]^{\pi_1}\ar[d]_{\pi_2}&G_1\ar[d]^r\\
G_1\ar[r]_d&G_0}}\;.\]
We remark that, since $G$ is a groupoid rather than just an internal category, the \emph{multiplication map} $m$ is a pullback of $d$ along itself:
\[\vcenter{\xymatrix{G_2\ar[r]^{\pi_1}\ar[d]_{m}&G_1\ar[d]^d\\
G_1\ar[r]_d&G_0}}\;.\]
A localic groupoid $G$ is said to be \emph{open} if $d$ is an open map. Hence, if $G$ is open, $m$ is also an open map. An \emph{\'etale groupoid} is an open groupoid such that $d$ is a local homeomorphism, in which case all the structure maps are local homeomorphisms and, hence, $G_0$ is isomorphic to an open sublocale of $G_1$. Conversely, any open groupoid for which $u$ is an open map is necessarily \'etale \cite[Corollary 5.12]{Re07}.

Similar conventions and remarks apply to topological groupoids, which are the internal groupoids in $\Top$. We remark that, keeping with \cite{Re07,GSQS}, our usage of $d$ and $r$ is reversed with respect to the typical conventions for groupoid C*-algebras.

The category whose objects are the localic \'etale groupoids and whose morphisms are the internal functors in $\Loc$ will be denoted by $\Grpd$. The following proposition will be useful later on:

\begin{proposition}\label{laxfunlemma}
Let $G$ and $H$ be \'etale groupoids and let
\begin{eqnarray*}
f_0:G_0\to H_0\\
f_1:G_1\to H_1
\end{eqnarray*}
be two maps of locales that satisfy the following properties:
\begin{eqarray}
f_1\circ i &=&i\circ f_1\;;\label{fun-i}\\
f_1\circ u &=& u\circ f_0\;;\label{fun-u}\\
f_0\circ d &=& d\circ f_1\;;\label{fun-d}\\
m\circ(f_1\tp f_1) &\le & f_1\circ m\;.\label{fun-m}
\end{eqarray}%
Then the pair $(f_0,f_1)$ is a functor of groupoids.
\end{proposition}

\begin{proof}
All we have to do is prove that the above inequality is in fact an equality. In point-set notation this follows from a simple series of inequalities:
\begin{eqnarray*}
f_1(x)f_1(y)&\le& f_1(xy)=f_1(xy)f_1(y)^{-1}f_1(y)
=f_1(xy)f_1(y^{-1})f_1(y)\\
&\le& f_1(xyy^{-1})f_1(y)=f_1(x)f_1(y)\;.
\end{eqnarray*}
Converting this to an explicit argument about locale maps is tedious but straightforward. \qed
\end{proof}

\subsection{Quantales}\label{quantalesandmodules}

By an \emph{involutive quantale} $Q$ is meant an involutive semigroup in $\SL$. In particular, the multiplication is a sup-lattice homomorphism
$\mu:Q\tp Q\to Q$, and
we shall adopt the following terminology and notation:

\begin{itemize}
\item The \emph{product} $\mu(a\tp b)$ of two elements $a,b\in Q$ is denoted by $ab$.
\item The involute of an element $a\in Q$ is denoted by $a^*$.
\item The involutive quantale $Q$ is \emph{unital} if there is a unit for the multiplication, which is denoted by $e_Q$ or simply $e$.
\item By a \emph{homomorphism} of involutive quantales $h:Q\to R$ is meant a homomorphism of involutive semigroups in $\SL$. If $Q$ and $R$ are unital, the homomorphism $h$ is \emph{unital} if $h(e_Q)=e_R$.
\end{itemize}

Given a unital involutive quantale $Q$, by a \emph{(left) $Q$-module} will be meant a sup-lattice $M$ equipped with a unital associative left action $Q\otimes M\to M$ in $\SL$ (the involution of $Q$ plays no role). The action of an element $a\in Q$ on $x\in M$ is denoted by $ax$ or, sometimes, for the sake of clarity, $a\act x$. By a \emph{homomorphism} of left $Q$-modules $h:M\to N$ is meant a $Q$-equivariant homomorphism of sup-lattices.

An involutive quantale can be associated to any open localic groupoid $G$ because,
since the multiplication map $m$ is open, there is a sup-lattice homomorphism defined as the following composition (in $\SL$):
\[\xymatrix{G_1 \tp G_1\ar@{->>}[r]& G_2\ar[r]^-{m_!}& G_1}\;.\]
This defines an associative multiplication on $G_1$ which, together with the isomorphism $G_1\stackrel{i_!}\to G_1$, turns $G_1$ into an involutive quantale $\opens(G)$ --- the ``opens of $G$''. This is unital if and only if $G$ is \'etale \cite[Corollary 5.12]{Re07}, in which case the unit is $e=u_!(1)$ and $u_!$ defines an order-isomorphism
$u_!: G_0\stackrel\cong\longrightarrow
\downsegment(e)$.
Hence, in particular, $\downsegment(e)$ is a frame.

The quantales associated in this way to \'etale groupoids are the \emph{inverse quantal frames} \cite{Re07}. They are precisely the unital involutive quantales $Q$ that are also frames and for which the following properties hold:
\begin{enumerate}
\item $a1\wedge e= aa^*\wedge e$ for all $a\in Q$;
\item $(a1\wedge e)a=a$ for all $a\in Q$;
\item $\V Q_\ipi=1$, where $Q_\ipi=\{s\in Q\st ss^*\vee s^*s\le e\}$ is the set of \emph{partial units} of $Q$.
\end{enumerate}
We note that $Q_\ipi$ is an infinitely distributive inverse semigroup (see \cite{Lawson}) whose idempotents are such that $E(Q_\ipi)=\downsegment(e)$. The latter is called the \emph{base locale} of $Q$ and we denote it by $Q_0$. For all $b\in Q_0$ and $a\in Q$ we have
\begin{eq}\label{iqfprop}
ba = b1\wedge a \textrm{ and } ab=1b\wedge a\;.
\end{eq}%
We also have
\begin{eq}\label{lcc}
Q\cong\lcc(Q_\ipi)\;,
\end{eq}%
where the right hand side is the join-completion of $Q_\ipi$ that preserves the joins of compatible sets (a subset $S$ of an inverse semigroup is \emph{compatible} if for all $s,t\in S$ both $st^*$ and $s^*t$ are idempotents --- see \cite{Lawson}).

A converse construction exists that assigns a localic \'etale groupoid $\groupoid(Q)$ to each inverse quantal frame $Q$, and we have, for all \'etale groupoids $G$ and all inverse quantal frames $Q$, an equivalence as follows \cite{Re07}:
\begin{eq}\label{gpdiqfequiv}
\opens(\groupoid(Q))=Q\ , \ \ \ \ \ \ \groupoid(\opens(G))\cong G\;.
\end{eq}%
Similarly, the topology $\topology(G_1)$ of a topological \'etale groupoid $G$ is an inverse quantal frame (and a spatial locale).
The \emph{category of inverse quantal frames} $\InvQuF$ \cite{Re07} has the homomorphisms of unital involutive quantales as morphisms. There are no other unital homomorphisms:

\begin{proposition}\label{necessarilyinv}
Any homomorphism of unital quantales $h:Q\to R$ between inverse quantal frames is necessarily involutive.
\end{proposition}

\begin{proof}
Since $h$ is unital it restricts to a homomorphism of inverse semigroups $Q_\ipi\to R_\ipi$. This necessarily preserves inverses and, since every element of an inverse quantal frame is a join of partial units, the conclusion follows. \qed
\end{proof}

\subsection{Actions}\label{secgrpdact}

Let $G$ be an \'etale groupoid. A \emph{left $G$-locale} $(X,p,\mathfrak a)$ consists of a map of locales $p:X\to G_0$, called the \emph{anchor map}, together with a map of locales
\[
\mathfrak a: G_1\tp_{G_0} X\to X\;,
\]
called the \emph{action}, where $G_1\tp_{G_0} X$ is the pullback of $r$ and $p$ in $\Loc$, satisfying the axioms for actions of internal categories, such as associativity (see \eg\ \cite[section 3.1]{GSQS}). The structure $(X,p,\mathfrak a)$ will often be denoted simply by $(X,p)$, or only $X$, when no confusion will arise. A right $G$-locale is defined similarly, with $X\tp_{G_0} G_1$ being the pullback of $p$ and $d$ in $\Loc$.
The category of left $G$-locales and equivariant maps between them is denoted by $G$-$\Loc$.

Similarly to the passage from \'etale groupoids to quantales, each left $G$-locale yields a left $\opens(G)$-module whose action is the sup-lattice homomorphism defined by the direct image of $\mathfrak a$ (which exists because $\mathfrak a$ is a pullback of $d$ along $p$):
\[\xymatrix{G_1 \tp X\ar@{->>}[r]& G_1\tp_{G_0} X\ar[r]^-{{\mathfrak a}_!}& X}\;.\]
In order to simplify notation let us write $Q$ instead of $\opens(G)$. We denote the left $Q$-module associated to a left $G$-locale $(X,p,\mathfrak a)$ by $X$ (rather than $\opens(X)$ as in \cite{GSQS}). This is a \emph{(left) $Q$-locale}, by which is meant a locale $X$ that is also a unital left $Q$-module satisfying the following \emph{anchor condition} for all $b\in Q_0$ and $x\in X$:
\begin{eq}\label{Qlocaleproperty}
bx=b1\land x\;.
\end{eq}%
The \emph{category of left $Q$-locales} \cite{GSQS} has the left $Q$-locales as objects, and the morphisms are the maps of locales whose inverse images are homomorphisms of left $Q$-modules. This category is denoted by $Q$-$\Loc$ and it is isomorphic to $G$-$\Loc$. The following equivalent formulas for the inverse image of the action will be needed later on:
\begin{eqarray}
{\mathfrak a}^*(x) &=& \V\{a\tp y\in Q\tp_{Q_0} X\st ay\le x\}\label{rightadjoint}\;;\\
{\mathfrak a}^*(x)&=&\V_{s\in Q_\ipi} s\otimes s^*x\;.\label{alphastareq}
\end{eqarray}%
Similar facts hold for right actions:
\begin{eqarray}
{\mathfrak a}^*(x) &=& \V\{y\tp a\in X\tp_{Q_0} Q\st ya\le x\}\label{rightadjointright}\;;\\
{\mathfrak a}^*(x)&=&\V_{s\in Q_\ipi} xs^*\tp s\;.\label{alphastareqright}
\end{eqarray}%
We conclude this overview of groupoid actions by looking at a few simple properties of $Q$-locales.
Eq.\ (\ref{Qlocaleproperty}) immediately implies both distributivity and ``middle-linearity'' of the action of the locale $Q_0$ over binary meets, for all $b\in Q_0$ and $x,y\in X$:
\begin{eqarray}
&b(x\wedge y)=b1\wedge x\wedge y=(b1\wedge x)\wedge(b1\wedge y)=bx\wedge by\;; \label{Bdist}\\
&bx\wedge y=bx\wedge by=x\wedge by\;.
\end{eqarray}%
Generalizing this to partial units we obtain:
\begin{proposition}\label{littlelemma}
Let $X$ be a $Q$-locale. For all $s\in Q_\ipi$ and $x,y\in X$, we have
\begin{enumerate}
\item $s(x\wedge y) = sx\wedge sy$, \label{dist1}
\item $s(x\wedge s^* y) = sx\wedge y$. \label{dist2}
\end{enumerate}
\end{proposition}

\begin{proof}
The inequality $s(x\wedge y)\le sx\wedge sy$ follows immediately from the monotonicity of the action. For the converse inequality, we use the distributivity \ref{Bdist}, with $b=ss^*$, in order to prove \ref{dist1}:
\begin{eqnarray*}
sx\wedge sy &=& ss^*sx\wedge ss^*sy = ss^*(sx\wedge sy)\\
&\le& s(s^*sx\wedge s^*sy)=ss^*s(x\wedge y)\\
&=&s(x\wedge y)\;.
\end{eqnarray*}
Condition \ref{dist2} follows easily: $sx\wedge y = ss^*sx\wedge y=sx\wedge ss^* y=s(x\wedge s^* y)$. \qed
\end{proof}

\section{Functoriality I}\label{funct1}

We begin by briefly addressing the extent to which the correspondence between \'etale groupoids and quantales is functorial with respect to groupoid functors, going a bit beyond \cite{Re07} by showing that, although the assignment from \'etale groupoids to quantales is not functorial unless quantale homomorphisms are ``lax'', the assignment from inverse quantal frames to groupoids is. A similar fact has been noticed in \cite{LL}, in the context of topological groupoids and inverse semigroups.

\subsection{Group homomorphisms}\label{grouphoms}

A similar discrepancy to the one we alluded to in the introduction occurs when relating localic groupoids and quantales. On one hand, the (tautological) functor from $\Loc$ to $\Frm$ is contravariant, whereas, on the other hand, it is the covariant powerset functor (rather than the contravariant one) which gives us a functor from the category of groups to the category of unital involutive quantales.
More than that, the covariant powerset functor is left adjoint to the functor that to each unital quantale $Q$ assigns its groups of units
\[Q^{\times}=\{a\in Q\st ab=e\textrm{ for some }b\in Q\}\;,\]
and thus the group homomorphisms can be identified with homomorphisms of unital quantales:
\[
\hom(G,H)\cong\hom(\pwset G,\pwset H)\;.
\]
Moreover, for each discrete group $G$ we have $G\cong\pwset G^{\times}$ (the adjunction is a co-reflection).

On the contrary, the contravariant powerset functor behaves poorly with respect to group homomorphisms:

\begin{lemma}
The homomorphisms $f:G\to H$ of discrete groups whose inverse image mappings $f^{-1}:\pwset H \to\pwset G$
are homomorphisms of unital quantales are precisely the isomorphisms.
\end{lemma}

\begin{proof}
Let $f:G\to H$ be a homomorphism of discrete groups.
If $f^{-1}$ is a homomorphism of quantales and $h\in H$ we have
\[
f^{-1}(\{h\})f^{-1}(\{h^{-1}\})=f^{-1}(\{h\}\{h^{-1}\})=\ker f\;.
\]
Therefore $1\in f^{-1}(\{h\})f^{-1}(\{h^{-1}\})$, which shows that $h\in f(G)$ and thus $f$ is surjective.

Conversely, if $f$ is surjective and $g\in f^{-1}(\{h\}\{k\})$ there is $k_0\in G$ such that $f(k_0)=k$ and, setting $g_1=gk_0^{-1}$ and $g_2=k_0$, we have \[g=g_1g_2 \textrm{ and }f(g_1)=h\textrm{ and }f(g_2)=k\;,\]
whence $g\in f^{-1}(\{h\})f^{-1}(\{k\})$. This shows that
\[f^{-1}(\{h\}\{k\})\subset f^{-1}(\{h\})f^{-1}(\{k\})\;,\]
and thus $f^{-1}$ is a homomorphism of quantales.

Finally, the quantale unit is preserved by $f^{-1}$ if and only if $\{1\}=\ker f$, \ie, $f$ is injective.\qed
\end{proof}

This shows that in order to obtain a contravariant functor to the category of unital quantales from a category of \'etale groupoids whose morphisms are functors,  we should either enlarge the class of quantale homomorphisms or severely restrict the class of groupoid functors.

\subsection{Covering functors}\label{quantalfunctors}

The idea of restricting the class of groupoid functors has been adopted by Lawson and Lenz \cite{LL}, who have shown, in the context of topological \'etale groupoids, that the notion of covering functor (as in \citelist{\cite{BrownTopGpds}*{sec.\ 10.2} \cite{GabrielZisman}*{p.\ 139}}) is equivalent to that of a functor $(f_0,f_1):G\to H$ such that \[f_1^{-1}:\topology(H_1)\to\topology(G_1)\] is a homomorphism of unital quantales \cite[Lemma 2.20]{LL}. In this section we see that any homomorphism $h$ of unital quantales between inverse quantal frames equals an inverse image $f_1^*$ for a localic groupoid functor $(f_0,f_1)$ if and only if $h$ preserves finite meets, which gives us a way of extending the definition of covering functor to localic groupoids, as we now explain.

\begin{definition}
\begin{enumerate}
\item The category $\IQFrm$ is the subcategory of $\InvQuF$ with the same objects and whose homomorphisms also preserve finite meets.
\item We denote the dual category $\opp{\IQFrm}$ by $\IQLoc$.
\end{enumerate}
\end{definition}

\begin{theorem}
The assignment $Q\mapsto\groupoid(Q)$ from inverse quantal frames to \'etale groupoids extends to a functor
\[\groupoid:\IQLoc\to \Grpd\;.\]
\end{theorem}

\begin{proof}
Let $Q$ and $R$ be inverse quantal frames and let $f:Q\to R$ be a morphism in $\IQLoc$. Writing $G$ and $H$ for $\groupoid(Q)$ and $\groupoid(R)$, respectively, we have, as locales, $G_1=Q$, $H_1=R$, $G_0=Q_0$ and $H_0=R_0$, with the structure maps of $G$ given in terms of the quantale structure by, for all $a\in G_1$ and $b\in G_0$,
\begin{eqnarray*}
u^*(a)&=&a\wedge e\\
d^*(b)&=&b1\\
i^*(a)&=& a^*\;.
\end{eqnarray*}
For $H$ it is similar, and we shall use the same notation for the structure maps of $H$, without any indices.
As a candidate for a groupoid functor we set $f_1=f$, and $f_0$ is given by defining $f_0^*$ to be the restriction of $f^*$ to $H_0$. (This is well defined because $f^*$ is unital.) We note that since $f^*$ preserves the quantale involution we immediately obtain
\begin{eq}\label{funct-i}
f_1\circ i=i\circ f_1\;.
\end{eq}%
Now let us prove the following equalities:
\begin{eqarray}
f_1\circ u &=& u\circ f_0\;;\label{funct-u}\\
f_0\circ d &=& d\circ f_1\;.\label{funct-d}
\end{eqarray}%
We have: for all $a\in H_1$
\[
f_0^*(u^*(a))=f^*(a\wedge e)=f^*(a)\wedge e=u^*(f_1^*(a))\;,
\]
which proves Eq.\ (\ref{funct-u});
for all $b\in H_0$
\[
d^*(f_0^*(b))=d^*(f^*(b))=f^*(b)1=f^*(b1)=f_1^*(d^*(b))\;,
\]
which proves Eq.\ (\ref{funct-d}).
By \cite[Lemma 5.13]{Re07} we have
\[
(f^*\tp f^*)\circ m^*\le m^*\circ f^*\;,
\]
and thus by \ref{laxfunlemma} the pair $(f_0,f_1)$ is a groupoid functor. Finally, the assignment \[f\mapsto (f_0,f_1)\] is clearly functorial. \qed
\end{proof}

\begin{definition}
The category $\CGrpd$ is the subcategory of $\Grpd$ whose morphisms are the \emph{covering functors}, by which we mean the continuous functors $f:G\to H$ such that \[f_1^*:\opens(H)\to\opens(G)\] is a homomorphism of unital involutive quantales (\ie, $f_1$ is a morphism in $\IQLoc$).
\end{definition}

\begin{corollary}
The categories $\CGrpd$ and $\IQLoc$ are equivalent.
\end{corollary}

\begin{proof}
For each inverse quantal frame $Q$ we have $\opens(\groupoid(Q))=Q$.
And for each \'etale groupoid $G$ we have $\groupoid(\opens(G))\cong G$, where a canonical isomorphism $\iota_G=(\iota_0,\iota_1):G\to\groupoid(\opens(G))$ in $\Grpd$ is such that $\iota_1$ is the identity on $G_1$ and
\[
\iota_0:G_0\to Q_0
\]
is the codomain restriction of $u_!:G_0\to G_1$. The two assignments $Q\mapsto\groupoid(Q)$ and $G\mapsto\opens(G)$, which extend to functors as we have seen, together with the two natural transformations
\begin{eqnarray*}
\ident&:& I\Rightarrow \opens\circ\groupoid\\
\iota&:& \groupoid\circ\opens\Rightarrow I\;,
\end{eqnarray*}
yield an adjoint equivalence of categories. \qed
\end{proof}
 
 \subsection{Lax homomorphisms}

For the sake of completeness let us take a very brief look at an alternative way of obtaining functoriality ``on the nose'', namely by enlarging the class of quantale homomorphisms.

We write $\IQLocLax$ for the extension of $\IQLoc$ whose objects are the inverse quantal frames and whose morphisms \[f:R\to Q\] are the maps of locales such that
\begin{eqnarray*}
f^*(a)f^*(b)&\le& f^*(ab)\ \ \ \ \textrm{for all }a,b\in Q\;,\\
f^*(a^*)&=& f^*(a)^*\;,\\
e_R&\le& h(e_Q)\;.
\end{eqnarray*}%

\begin{theorem}
The assignment $G\mapsto\opens(G)$ extends to a faithful functor
\[\opens:\Grpd\to\IQLocLax\;.\]
\end{theorem}

\begin{proof}
Let $f:G\to H$ be a morphism of $\Grpd$, and let $Q=\opens(G)$ and $R=\opens(H)$. The assignment $f\mapsto f_1^*$ is of course functorial and faithful, so we only have to verify that $f_1^*:R\to Q$ satisfies the three above conditions. The first is a consequence of \cite[Lemma 5.13]{Re07}, and the second is an immediate consequence of the fact that functors preserve inverses. The third also holds, as we now explain. The axiom \[u_H\circ f_0= f_1\circ u_G\] of groupoid functors implies
\[
f_0^*\circ u^*_H\le u^*_G\circ f_1^*
\]
which, by adjointness, gives us
\[
(u_G)_!\circ f_0^*\circ u^*_H\le f_1^*\;.
\]
Composing with $(u_H)_!$ we obtain
\[
(u_G)_!\circ f_0^*\circ u^*_H\circ (u_H)_!\le f_1^*\circ (u_H)_!\;,
\]
and this, using the unit of the adjunction $\ident\le u^*_H\circ (u_H)_!$, implies
\[
(u_G)_!\circ f_0^*\le f_1^*\circ (u_H)_!\;.
\]
Hence,
\[e_Q=(u_G)_!(1_{G_0})= (u_G)_!(f_0^*(1_{H_0}))\le f_1^*( (u_H)_!(1_{H_0}))=f_1^*(e_R)\;. \qed\]
\end{proof}

\section{Groupoid actions}\label{gpdactions}

Let us study some constructions related to orbits of groupoid actions, in the language of quantale modules.

\subsection{Orbits}

If $G$ is an \'etale groupoid and $X$ is a left $G$-locale, we can construct the \emph{orbit locale} of the action as the coequalizer in $\Loc$
\[
\xymatrix{
G_1\tp_{G_0} X\ar@<-1.2ex>[rr]_-{\pi_2}\ar@<1.2ex>[rr]^-{\mathfrak a}&&X\ar[r]&X/G\;.
}
\]
The locale points of $X/G$ can be regarded as being the orbits of the action of $G$ on $X$. 

\begin{definition}
We refer to $X/G$ as the \emph{quotient of $X$ by $G$}. For a right $G$-locale the corresponding quotient is denoted by $G\backslash X$.
\end{definition}

There is a simple description of these quotients in terms of $\opens(G)$-modules. We explain this for left actions only, as for right actions everything is similar.

\begin{definition}\label{invariant}
Let $G$ be an \'etale groupoid with quantale $Q=\opens(G)$, and $X$ a left $G$-locale. An \emph{element} $x\in X$ is \emph{invariant} if the following equivalent conditions hold (for $X$ regarded as a $Q$-module):
\begin{enumerate}
\item For all $a\in Q$ we have $ax\le x$;
\item For all $s\in Q_\ipi$ we have $sx\le x$;
\item $1x\le x$;
\item $1x=x$.
\end{enumerate}
\end{definition}

\begin{theorem}
Let $G$ be an \'etale groupoid and $X$ a left $G$-locale. The quotient $X/G$
coincides with the set of invariant elements of the action.
\end{theorem}

\begin{proof}
First we remark that the invariant elements form an obvious subframe $F\subset X$, hence defining a quotient locale as required. It remains to be shown that the following diagram is an equalizer in the category of sets, where $\iota$ is the frame inclusion:
\[
\xymatrix{
F\ar[r]^-\iota&X\ar@<-1.2ex>[rr]_-{\pi_2^*}\ar@<1.2ex>[rr]^-{\mathfrak a^*}&&G_1\tp_{G_0} X\;.
}
\]
In other words, we need to show that $x$ is invariant if and only if
\begin{eq}
\pi_2^*(x)=\mathfrak a^*(x)\;.\label{invequiv}
\end{eq}%
Let us assume that Eq.\ (\ref{invequiv}) holds. Using the co-unit of the adjunction $\mathfrak a_!\dashv\mathfrak a^*$ we conclude that $x$ is invariant:
\[
1x=\mathfrak a_!(1\tp x)=\mathfrak a_!(\pi_2^*(x))
=\mathfrak a_!(\mathfrak a^*(x))\le x\;.
\]
Conversely, let us assume that $x$ is invariant. By Eq.\ (\ref{rightadjoint}), the condition $1x\le x$ immediately implies that $1\tp x\le\mathfrak a^*(x)$. And, by Eq.\ (\ref{alphastareq}), we have, writing $Q=\opens(G)$,
\[\mathfrak a^*(x) = \V_{s\in Q_\ipi} s\tp s^*x\le\V_{s\in Q_\ipi} s\tp x=1\tp x\;.\]
Hence, Eq.\ (\ref{invequiv}) holds. \qed
\end{proof}

We remark that, although this is not needed in what follows, the idea that the orbits must be certain ``subspaces'' can be explicitly conveyed by first observing that, as a subframe, $X/G$ is in fact closed under arbitrary meets in $X$, which means that it also defines a quotient of $X$ in $\SL$ \cite{JT}. This does not correspond to a sublocale of $X$ because the quotient is not taken in $\Frm$. However,
by freely adjoining finite meets to $X$ we obtain the \emph{lower powerlocale} $\lpl X$  (one of several localic notions of ``powerspace of $X$''), whose points can be identified (in an arbitrary topos) with the ``weakly closed sublocales of $X$ with open domain'' \cite{BuFu96} (and coincide, in classical set theory, with the closed sublocales of $X$ --- see \cite{RV}).
Hence, the sup-lattice quotient $X\to X/G$ extends uniquely to a frame quotient $\lpl X\to X/G$, hence depicting $X/G$ as a sublocale of $\lpl X$, and allowing us to view the orbits of the action as being sublocales of $X$.

\subsection{Diagonal actions}

Let $G$ be an \'etale groupoid with quantale $Q=\opens(G)$. Given right and left $G$-locales $(X,p,\mathfrak a)$ and $(Y,q,\mathfrak b)$, we can define on the pullback $X\tp_{G_0} Y$ of $p$ and $q$ (which equals $X\tp_{Q_0} Y$) the diagonal action which, in point-set notation, would be given by the formula
\[g\act(x,y)=(x\act g^{-1},g\act y)\;.\]
Module-theoretically this goes as follows:

\begin{theorem}
Let $G$ be an \'etale groupoid with quantale $Q=\opens(G)$, and let $X$ and $Y$ be a right $G$-locale and a left $G$-locale with anchor maps $p$ and $q$, respectively. The following conditions hold:
\begin{enumerate}
\item (\emph{Diagonal action}.) A left quantale action of $Q$ on $X\tp_{Q_0} Y$ is defined, for all $x\in X$, $y\in Y$ and $s\in Q_\ipi$, by the condition
\begin{eq}
s\act(x\tp y)=xs^*\tp sy\;.\label{action}
\end{eq}%
\item This action makes $X\tp_{Q_0} Y$ a left $Q$-locale.
\end{enumerate}
\end{theorem}

\begin{proof}
For each $s\in Q_\ipi$, let the mapping
\[f_s:X\oplus Y\to X\tp_{Q_0} Y\]
be defined by
\[f_s(x,y)=xs^*\otimes sy\;.\]
This clearly preserves joins in each variable separately. And, for each $b\in Q_0$, the following middle-linearity condition is satisfied:
\begin{eqnarray*}
f_s(xb,y)&=&xbs^*\otimes sy = xbs^*ss^*\otimes sy=xs^*sbs^*\otimes sy\\
&=&xs^*\otimes sbs^*sy=xs^*\otimes ss^*sby=xs^*\otimes sby\\
&=&f_s(x,by)\;.
\end{eqnarray*}
Hence, $f_s$ factors uniquely through the sup-lattice homomorphism given by
\[x\otimes y\mapsto xs^*\otimes sy\;,\]
and thus the semigroup $Q_\ipi$ acts by endomorphisms on $X\tp_{Q_0} Y$ (the associativity of the action is immediate). Now recall the isomorphism $Q\cong\lcc(Q_\ipi)$ of (\ref{lcc}) --- the right hand side is the frame of \emph{compatible ideals of $Q_\ipi$}, which are the downwards-closed subsets of $Q_\ipi$ that are closed under the formation of joins of compatible subsets. In order to show that the action of $Q_\ipi$ extends to the required action of $Q$ it suffices to show that the semigroup action respects such joins. Let then $Z\subset Q_\ipi$ be compatible, \ie, a subset such that for all $s,t\in Z$ we have $st^*\le e$ and $s^* t\le e$. Then $\V Z\in Q_\ipi$ and, for all $s,t\in Z$, $x\in X$ and $y\in Y$, we have
\[xs^*\tp ty=xs^*ss^*\tp ty=xs^*\tp ss^*t y\le xs^*\tp sy\;,\]
and thus we obtain
\begin{eqnarray*}
\left(\V Z\right)\act(x\tp y) &=&
x\left(\V Z\right)^*\tp \left(\V Z\right)y=\V_{s,t\in Z}xs^*\tp ty\\
&=& \V_{s\in Z} xs^*\tp sx=\V \left(Z\act(x\tp y)\right)\;.
\end{eqnarray*}
This proves that $X\tp_{Q_0} Y$ is a left $Q$-module with the action defined by Eq.\ (\ref{action}). And it is a $Q$-locale because the anchor condition holds: for all $b\in Q_0$ and $\xi=\V_i x_i\tp y_i\in X\tp_{Q_0} Y$ we have
\begin{eqnarray*}b\act\xi&=&\V_i x_i b\tp b y_i =\V_i (1b\wedge x_i)\tp(b1\wedge y_i)
=\V_i(1b\tp b1)\wedge(x_i\wedge y_i)\\
&=&b\act(1\tp 1)\wedge\xi\;. \qed
\end{eqnarray*}
\end{proof}

\subsection{Tensor products}

Let $G$ be an \'etale groupoid. Given right and left $G$-locales $(X,p,\mathfrak a)$ and $(Y,q,\mathfrak b)$, a \emph{tensor product over $G$} can be defined as a coequalizer in $\Loc$ (\cf\ \cite{Moer87,Moer90}):
\begin{eq}\label{tensorcoeq}
\xymatrix{
X\tp_{G_0} G_1\tp_{G_0} Y\ar@<-1.2ex>[rr]_-{\langle\pi_1,\mathfrak b\circ \pi_{23}\rangle}\ar@<1.2ex>[rr]^-{\langle\mathfrak a\circ\pi_{12},\pi_3\rangle}&&X\tp_{G_0} Y\ar[r]&X\tp_G Y\;.
}
\end{eq}%
Our aim now is to show that this tensor product coincides with the ``ring-theoretic'' tensor product of $\opens(G)$-modules, and our first step will be to show module-theoretically that $X\tp_G Y$ can be given an equivalent definition as the quotient $(X\tp_{G_0} Y)/G$ by the diagonal action (\cf\ \cite{Mr99}).

\begin{lemma}\label{Gtenseqquot}
Let $G$ be an \'etale groupoid, and $(X,p,\mathfrak a)$ and $(Y,q,\mathfrak b)$ a right and a left $G$-locale, respectively. Then
\[X\tp_G Y=(X\tp_{G_0} Y)/G\;.\]
\end{lemma}

\begin{proof}
The coequalizer $X\tp_G Y$ can be concretely identified with the subframe of $X\tp_{G_0} Y$ consisting of the elements $\xi$ such that
\[
[\pi_{12}^*\circ\mathfrak a^*,\pi_3^*](\xi)=[\pi_1^*,\pi_{23}^*\circ\mathfrak b^*](\xi)\;.
\]
Using Eqs.\ (\ref{alphastareq}) and (\ref{alphastareqright}), respectively for $\mathfrak b^*$ and $\mathfrak a^*$, this equality is equivalent, letting $\xi=\V_i x_i\tp y_i$ and writing $Q=\opens(G)$, to
\begin{eq}\label{equalizer}
\V_{i}\V_{s\in Q_\ipi} x_is^*\tp s\tp y_i=\V_i\V_{s\in Q_\ipi} x_i\tp s\tp s^* y_i\;.
\end{eq}%
In order to conclude the proof we show that $\xi$ satisfies this equality if and only if it is invariant with respect to the diagonal action. Let us assume that Eq.\ (\ref{equalizer}) holds. Then $\xi$ is invariant:
\begin{eqnarray*}
1\xi&=&\V_{i,s} x_i s^*\tp s y_i = \mathfrak a_!\tp\ident\left(\V_{i,s} x_i\tp s^*\tp sy_i\right)\\
&=&\mathfrak a_!\otimes \ident\left(\V_{i,s} x_i s\tp s^*\tp y_i\right)
=\V_{i,s} x_i ss^*\tp y_i=\xi\;.
\end{eqnarray*}
Conversely, assuming that $\xi$ is invariant, Eq.\ (\ref{equalizer}) holds:
\[\begin{array}{rcll}
\V_{i,s} x_i\tp s\tp s^* y_i &\le& \mathfrak a^*\tp\ident\left(\V_{i,s} x_i s\tp s^*y_i\right)
&\textrm{[by Eq.\ (\ref{rightadjointright})]}\\
&\le&\mathfrak a^*\tp\ident\left(\V_i x_i\tp y_i\right)&(s\act\xi\le\xi)\\
&=&\V_{i,s} x_i s^*\tp s\tp y_i & \textrm{[by Eq.\ (\ref{alphastareqright})]}\\
&\le&\ident\tp\mathfrak b^*\left(\V_{i,s} x_i s^*\tp s y_i\right)&\textrm{[by Eq.\ (\ref{rightadjoint})]}\\
&\le&\ident\tp\mathfrak b^*\left(\V_i x_i\tp y_i\right)&(s\act\xi\le \xi)\\
&=&\V_{i,s} x_i\tp s\tp s^* y_i &\textrm{[by Eq.\ (\ref{alphastareq})]}\;. \qed
\end{array}\]
\end{proof}

\begin{theorem}\label{thmtp}
Let $G$ be an \'etale groupoid, and let $X$ and $Y$ be a right $G$-locale and a left $G$-locale as in the previous lemma. Then,
\[X\tp_G Y = X\tp_{\opens(G)} Y\;.\]
\end{theorem}

\begin{proof}
Let us write $Q$ for $\opens(G)$. As a sup-lattice, $X\tp_Q Y$ is the quotient of $X\tp_{G_0} Y$ (which equals $X\tp_{Q_0} Y$) generated by the middle-linearity relations
\[xa\tp y=x\tp ay\]
for all $a\in Q$, and it is sufficient to take $a\in Q_\ipi$. By general sup-lattice algebra \cite{JT}, the sup-lattice quotient can be concretely identified with the subset of $X\tp_{G_0} Y$ whose elements $\xi$ are closed under the relations; that is, such that for all $x\in X$, $y\in Y$ and $s\in Q_\ipi$ we have
\begin{eq}
xs\tp y\le \xi\iff x\tp sy\le\xi\;.\label{midlin}
\end{eq}%
By \ref{Gtenseqquot}, $X\tp_G Y$ can be identified with the set of invariant elements for the action Eq.\ (\ref{action}), so let us show that the invariant elements are the same as those which satisfy the condition (\ref{midlin}). Let $\xi$ be an invariant element of $X\tp_{G_0} Y$, \ie, such that $s\act\xi\le\xi$ for all $s\in Q_\ipi$, and let $x\in X$, $y\in Y$, and $s\in Q_\ipi$. If $xs\tp y\le\xi$ we obtain
\[x\tp sy=x\tp ss^*s y=xss^*\tp sy=s\act(xs\tp y)\le s\act\xi\le\xi\;,\]
and, similarly, if $x\tp sy\le\xi$ we conclude $xs\tp y\le \xi$. Hence, $\xi$ satisfies (\ref{midlin}). For the converse, assume that $\xi=\V_i x_i\tp y_i$ satisfies (\ref{midlin}).
For all $i$ and $s\in Q_\ipi$, we have
\[x_i s^*s\tp y_i\le x_i\tp y_i\le\xi\]
and, using (\ref{midlin}),
\[s\act(x_i\tp y_i)=x_i s^*\tp sy_i\le\xi\;.\]
Hence, $\xi$ is invariant, and we conclude that $X\tp_G Y$ coincides, concretely as a subset of $X\tp_{G_0} Y$, with $X\tp_Q Y$. \qed
\end{proof}

\section{Functoriality II}\label{funct2}

Now we address the main aim of this paper, which is to show that groupoid bi-actions can be identified with a natural notion of bilocale for inverse quantal frames, and to establish an ensuing (bicategorical) equivalence between \'etale groupoids and inverse quantal frames. Following that, we discuss connections to algebraic morphisms of groupoids in the sense of \cite{BS05,Bun08}.

\subsection{Bimodules}

Let $Q$ and $R$ be unital quantales. By a \emph{$Q$-$R$-bimodule} is meant a sup-lattice $_Q X_R$, which can simply be denoted by $X$, equipped with structures of unital left $Q$-module and unital right $R$-module that satisfy the associativity condition
\[(rx) q=r(x q)\ \ \ \ \textrm{for all }r\in R,\ x\in X,\ q\in Q\;.\]
Similarly to rings, we obtain a bicategory \cite[sec.\ 2.5, 5.7]{Benabou}: the 0-cells are the unital quantales; the 1-cells are the bimodules $_Q X_R$; the composition of 1-cells $_Q X_R$ and $_R Y_S$ is given by $Y\circ X=X\tp_R Y$; and the 2-cells are the homomorphisms of bimodules, with composition defined as usual. A homomorphism of unital quantales $h:Q\to R$ can be identified with a $Q$-$R$-bimodule $X_h$, which is $R$ with the left $Q$-action induced by $h$ and the right $R$-action given by multiplication; there are canonical isomorphisms
\[
X_{h\circ k} \cong X_h\circ X_k\;,
\]
and the assignments $Q\mapsto Q$ and $h\mapsto X_h$ embed the category of unital quantales in the bicategory.

\begin{definition}\label{QRbilocale}
Let $Q$ and $R$ be inverse quantal frames. A \emph{$Q$-$R$-bilocale} is a bimodule $_Q X_R$ that is also a locale such that for all $b\in Q_0$, $c\in R_0$ and $x\in X$ the following \emph{left and right anchor conditions} hold:
\begin{eqarray}
bx &=& b1\land x\label{leftbc}\\
xc &=& 1c\land x\;.\label{rightbc}
\end{eqarray}%
A \emph{map} of bilocales $f:{_Q X_R}\to{_Q Y_R}$ is a map of locales whose inverse image $f^*$ is a homomorphism of bimodules, and the resulting category is denoted by $Q$-$R$-$\Loc$.
\end{definition}

It is immediate that any inverse quantal frame $Q$ is a $Q$-$Q$-bilocale, due to Eqs.\ (\ref{iqfprop}). In addition, bilocales behave well with respect to tensor products:

\begin{lemma}
Let $Q$, $R$, $S$ be inverse quantal frames. The tensor product $X\tp_R Y$ of bilocales $_Q X_R$ and $_R Y_S$ is a $Q$-$S$-bilocale.
\end{lemma}

\begin{proof}
$X\tp_R Y$ is a $Q$-$S$-bimodule, it is a locale due to \ref{thmtp}, and it is a bilocale because the left (and the right) anchor condition holds, since for all $b\in Q_0$, $x\in X$ and $y\in Y$ we have
\[b(x\tp y) = (b x)\tp y = (b1\wedge x)\tp y = (b1)\tp 1\wedge x\tp y = b(1\tp 1)\wedge x\tp y\;. \qed\]
\end{proof}

Hence, the following bicategory is well defined:

\begin{definition}
The bicategory $\biIQLoc$ has the inverse quantal frames as 0-cells, the bilocales as 1-cells, and the maps of bilocales as 2-cells. The composition of 1-cells $_Q X_R$ and $_R Y_S$ is defined by
\[Y\circ X = X\tp_R Y\;,\]
and the coherence isomorphisms are the maps of bilocales whose inverse images are coherence isomorphisms in the usual ``ring'' sense.
\end{definition}

\begin{lemma}\label{homsasbimodules}
The assignments $Q\mapsto Q$ and $h\mapsto X_h$ embed $\InvQuF$ into $\biIQLoc$.
\end{lemma}

\begin{proof}
All we have to do is prove that if $h:Q\to R$ is a morphism of $\InvQuF$ the bimodule $X_h$ is a bilocale rather than just a bimodule. It is a locale because $R$ is, the right anchor condition follows from Eqs.\ (\ref{iqfprop}), and the left anchor condition holds because $h$ is unital and thus $h(b)\le e$ for all $b\in Q_0$:
\[
b\act x=h(b)x=h(b)1\wedge x=b\act 1\wedge x\;. \qed
\]
\end{proof}

\subsection{Bi-actions}\label{sec:biact}

Let $G$ and $H$ be localic \'etale groupoids. A $G$-$H$-\emph{bilocale} is a locale $_G X_H$, which can be simply denoted by $X$, equipped with a left $G$-locale structure $(p,\mathfrak a)$ and a right $H$-locale structure $(q,\mathfrak b)$ such that the following diagrams in $\Loc$ are commutative:

\begin{eq}\label{biaxioms}
\begin{array}{c}
\xymatrix{G_1\tp_{G_0} X\ar[r]^-{\mathfrak a}\ar[d]_{\pi_2}&X\ar[d]^q\\
X\ar[r]_q&H_0}
\xymatrix{X\tp_{H_0} H_1\ar[r]^-{\mathfrak b}\ar[d]_{\pi_1}&X\ar[d]^p\\
X\ar[r]_p&G_0}\\
\xymatrix{G_1\tp_{G_0} X\tp_{H_0} H_1\ar[rr]^-{\mathfrak a\tp 1}\ar[d]_{1\tp \mathfrak b}&&X\tp_{H_0} H_1\ar[d]^{\mathfrak b}\\
G_1\tp_{G_0} X\ar[rr]_{\mathfrak a}&&X}
\end{array}
\end{eq}%
The first two diagrams assert that the anchor map of the $G$-locale is invariant under the action of $H$, and that the anchor map of the $H$-locale is invariant under the action of $G$. Both are in line with the idea that a bilocale may be regarded as being the graph of a binary relation between the ``orbit spaces'' of $G$ and $H$, and they ensure that the third diagram (associativity) makes sense.

A \emph{map} of bilocales $f:{_G X_H}\to {_G Y_H}$ is a map of locales that is both a map of left $G$-locales and a map of right $H$-locales. The resulting category of bilocales is denoted by $G$-$H$-$\Loc$. The maps of bilocales are the 2-cells of a bicategory, denoted by $\biGrpd$, whose 0-cells are the \'etale groupoids and whose 1-cells are the $G$-$H$-bilocales. The composition of 1-cells is defined by the tensor product: given 1-cells $_G X_H$ and $_H Y_K$ we define
\[
Y\circ X = X\tp_H Y\;.
\]
The coherence isomorphisms are standard (\cf\ \cite{Moer87,Moer90}).

\begin{theorem}\label{catiso}
Let $G$ and $H$ be \'etale groupoids. The categories $G$-$H$-$\Loc$ and $\opens(G)$-$\opens(H)$-$\Loc$ are isomorphic.
\end{theorem}

\begin{proof}
Let us denote $\opens(G)$ and $\opens(H)$ by $Q$ and $R$, respectively. Any bilocale $_G X_H$ has both a left $Q$-locale structure and a right $R$-locale structure, and it is a routine matter to verify that it is a $Q$-$R$-bilocale because the associativity condition,
\begin{eq}\label{bimoduassoc}
(ax)b = a(xb)
\end{eq}%
for all $a\in Q$, $x\in X$, and $b\in R$,
is essentially the direct image version of the associativity diagram of (\ref{biaxioms}):
\begin{eq}\label{dirimassoc}
\vcenter{\xymatrix{Q\tp_{Q_0} X\tp_{R_0} R\ar[rr]^-{\mathfrak a_!\tp \ident}\ar[d]_{\ident\tp \mathfrak b_!}&&X\tp_{R_0} R\ar[d]^{\mathfrak b_!}\\
Q\tp_{Q_0} X\ar[rr]_{\mathfrak a_!}&&X}}
\end{eq}%
Moreover, from the general results on groupoid actions (\cf\ section \ref{secgrpdact}) it follows that a map of locales $f:X\to Y$ between bilocales $_G X_H$ and $_G Y_H$ is a morphism in $G$-$H$-$\Loc$ if and only if it is a morphism in $Q$-$R$-$\Loc$. Therefore, all
that we have left to prove is that every $Q$-$R$-bilocale arises from a (necessarily unique) $G$-$H$-bilocale; that is, that the unique $G$-locale and $H$-locale structures obtained from the $Q$-locale and $R$-locale structures of a $Q$-$R$-bilocale further satisfy the commutativity of the three bilocale diagrams of (\ref{biaxioms}). Let $_Q X_R$ be a bilocale, and let $(p,\mathfrak a)$ and $(q,\mathfrak b)$ be, respectively, the unique $G$-locale and $H$-locale structures that it determines. The commutativity of the third diagram of (\ref{biaxioms}) follows from reversing the previous argument for associativity: it follows from the commutativity of (\ref{dirimassoc}), which is equivalent to the bimodule associativity. This kind of argument does not work for the first two diagrams of (\ref{biaxioms}) because we are not assuming that $p$ and $q$ are open maps, but we can nevertheless establish their commutativity in terms of inverse images of the locale maps.
Let us do this only for the first one,
\begin{eq}\label{repetition}
\vcenter{\xymatrix{G_1\tp_{G_0} X\ar[r]^-{\mathfrak a}\ar[d]_{\pi_2}&X\ar[d]^q\\
X\ar[r]_q&H_0}}
\end{eq}%
since the second is proved similarly. 
Recall \cite{GSQS} that the direct image of the open map $u:H_0\to H_1$ restricts to an order isomorphism $u_!:H_0\to R_0$ such that the following triangle commutes in $\SL$:
\[
\xymatrix{
&X\\
H_0\ar[ur]^{d^*}\ar[rr]_{u_!}&&R_0\ar[ul]_{1_X\act(-)}\;.
}
\]
Hence, the commutativity of (\ref{repetition}) is equivalent to the commutativity of the diagram in $\Frm$
\[
\vcenter{\xymatrix{
G_1\tp_{G_0} X&X\ar[l]_-{\mathfrak a^*}\\
X\ar[u]^{\pi_2^*}&R_0\ar[l]^{1_X\act(-)}\ar[u]_{1_X\act(-)}\;,
}}
\]
which commutes if and only if for all $c\in R_0$ we have
\[\mathfrak a^*(1_X\act c)=1_Q\otimes (1_X \act c)\;.\]
And the latter condition holds because, on one hand, from Eq.\ (\ref{rightadjoint})  and the equality  \[1_Q 1_X c = 1_X c\]
we obtain
\[1_Q\tp 1_X c\le \mathfrak a^*(1_X c)\;;\]
and, on the other, from Eq.\ (\ref{alphastareq}) we obtain
\[\mathfrak a^*(1_X c) = \V_{s\in Q_\ipi} s\otimes s^* 1_X c\le 1_Q\otimes 1_X c\;. \qed\]
\end{proof}

\begin{corollary}
The bicategories $\biGrpd$ and $\biIQLoc$ are biequivalent.
\end{corollary}

\subsection{Algebraic morphisms}\label{algebraicmorphisms}

Due to the biequivalence, the embedding $\InvQuF\to\biIQLoc$ yields a further embedding
\[
\InvQuF\to\biGrpd
\]
such that each homomorphism of inverse quantal frames $h:Q\to R$ maps to a $\groupoid(Q)$-$\groupoid(R)$-bilocale. Such a bilocale is precisely the same as
an \emph{algebraic morphism} of groupoids in the sense of Buneci and Stachura \cite{BS05,Bun08}.
Moreover, their composition of algebraic morphisms is, up to coherence, the same as that which results from the embedding. But it is strictly associative and therefore defines a category. The definitions can be carried over to very general groupoids:

\begin{definition}(Based on \cite{BS05}.)
Let $G$ and $H$ be groupoids. By an \emph{algebraic morphism} from $G$ to $H$ is meant a left action of $G$ on $H_1$ that commutes with right multiplication in $H$. More precisely, an algebraic morphism
\[
(p,\mathfrak a): G\to H
\]
consists of maps $\mathfrak a:G_1\tp_{G_0} H_1\to H_1$ and $p:H_1\to G_0$ that define a left $G$-locale and make the following diagrams commute:
\begin{eq}\label{amaxioms}
\begin{array}{c}
\xymatrix{G_1\tp_{G_0} H_1\ar[r]^-{\mathfrak a}\ar[d]_{\pi_2}&H_1\ar[d]^r\\
H_1\ar[r]_r&H_0}
\xymatrix{H_1\tp_{H_0} H_1\ar[r]^-m\ar[d]_{\pi_1}&X\ar[d]^p\\
H_1\ar[r]_p&G_0}\\
\xymatrix{G_1\tp_{G_0} H_1\tp_{H_0} H_1\ar[rr]^-{\mathfrak a\tp 1}\ar[d]_{1\tp m}&&H_1\tp_{H_0} H_1\ar[d]^m\\
G_1\tp_{G_0} H_1\ar[rr]_{\mathfrak a}&&H_1}
\end{array}
\end{eq}%
Given two algebraic morphisms
\[
G\stackrel{(p,\mathfrak a)}\longrightarrow H\stackrel{(q,\mathfrak b)}\longrightarrow K\;,
\]
their composition is
\[(q,\mathfrak b)\circ(p,\mathfrak a) = (p\circ u_H\circ q,\mathfrak c)\;,\]
where $\mathfrak c:G_1\tp_{G_0} K_1\to K_1$ is defined, in point-set notation, by
\[
g\act_{\mathfrak c} k = (g\act_{\mathfrak a} u_H(q(k)))\act_{\mathfrak b} k\;.
\]
\end{definition}

This definition applies to internal groupoids in any category with enough pullbacks, for instance topological or localic groupoids, Lie groupoids or, as in \cite{BS05,Bun08}, locally compact groupoids equipped with Haar systems of measures. For localic \'etale groupoids, denoting the resulting category by $\AGrpd$, we therefore conclude:

\begin{theorem}\label{equivalginvquf}
$\AGrpd$ and $\InvQuF$ are equivalent categories.
\end{theorem}

An immediate consequence of this equivalence is, of course, that algebraic morphisms specialize covariantly to homomorphisms of discrete groups and, contravariantly, to locale homomorphisms, as was stated by Buneci and Stachura, whose main goal was to define a (covariant) functor from groupoids to the category of C*-algebras \cite{BS05}: their functor assigns to each $\sigma$-locally compact  Hausdorff groupoid $G$ (equipped with a Haar system of measures) the multiplier algebra of the C*-algebra that arises as the completion of $C_c(G)$ with respect to a norm which is different from either the usual --- maximum or reduced --- ones. By restricting to \'etale groupoids (with counting measures) one obtains, due to \ref{equivalginvquf}, a non-trivial example of a functor from quantales to C*-algebras:

\begin{corollary}
There is a covariant functor from the full subcategory of $\InvQuF$ whose objects are the topologies of the $\sigma$-locally compact Hausdorff {\'e}tale groupoids to the category of C*-algebras and $*$-homomorphisms.
\end{corollary}

Algebraic morphisms have also been used by Buss, Exel and Meyer \cite{BEM12} in order to define a covariant functor from topological \'etale groupoids to inverse semigroups. Their functor can be identified, due to \ref{equivalginvquf}, with the covariant partial units functor $\ipi$ from spatial inverse quantal frames to inverse semigroups, and therefore it readily extends to localic groupoids.

The results in this section show that for \'etale groupoids the algebraic morphisms are subsumed by quantale homomorphisms. It is interesting to note that, albeit under completely different terminology, and restricting to discrete groupoids, the idea of defining a morphism of groupoids $G\to H$ to be a homomorphism of quantales $\opens(G)\to\opens(H)$ can be found in the work of Zakrzewski \cite{Zak90}, whose notion of pseudospace (\cf\ \cite{Wo80}) is based on the idea of replacing the underlying linear space of an associative $*$-algebra by the sup-lattice structure of a powerset, hence leading to algebras that are unital involutive quantales and furthermore, as the author states, are equivalent to discrete groupoids.
There is more than one way in which such ideas can be carried over to more general groupoids. For arbitrary open groupoids \cite{PR12} a definition of morphism $G\to H$ can of course be based on a homomorphism of involutive quantales $\opens(G)\to\opens(H)$, but additional requirements are needed, in particular
due to the absence of multiplicative units. Besides, for groupoids equipped with non-trivial additional structure, such as non-\'etale Lie groupoids, a homomorphism of quantales only takes the structure of topological groupoid into account. By contrast, algebraic morphisms were proposed by Buneci and Stachura precisely as a way of generalizing Zakrewski's ideas to (not necessarily \'etale) locally compact groupoids, and Zakrzewski's own extension to the differential setting \cite{Zak90b} defines morphisms $G\to H$ of Lie groupoids and symplectic groupoids to be ``differential relations'', \ie, submanifolds of $G_1\times H_1$ satisfying suitable conditions. Our results show that, nevertheless, the identification of algebraic morphisms with quantale homomorphisms is meaningful at least for \'etale groupoids.

~\\
{\sc\small
Center for Mathematical Analysis, Geometry and Dynamical Systems\\ Mathematics Department\\
Instituto Superior T\'{e}cnico, Universidade de Lisboa\\
Av.\ Rovisco Pais 1, 1049-001 Lisboa, Portugal}\\
{\it E-mail:} {\sf pmr@math.tecnico.ulisboa.pt}
\end{document}